\documentclass[
final
 , nomarks
]{dmtcs-episciences}

\usepackage[utf8]{inputenc}
\usepackage{subfigure}

\usepackage{mathtools}
\usepackage{amsthm} 

\theoremstyle{plain}
\newtheorem{theorem}{Theorem}
\newtheorem{proposition}[theorem]{Proposition}
\newtheorem{lemma}[theorem]{Lemma}
\newtheorem{corollary}[theorem]{Corollary}

\theoremstyle{definition}
\newtheorem{definition}[theorem]{Definition}
\newtheorem{remark}[theorem]{Remark}
\newtheorem{notation}[theorem]{Notation}
\theoremstyle{remark}
\newtheorem{claim}[theorem]{Claim}

\usepackage[round]{natbib}

\author{Narda Cordero-Michel
\and Hortensia Galeana-S{\'a}nchez }
\title[New Bounds for the Dichromatic Number]{New Bounds for the Dichromatic Number of a Digraph \thanks{This research was supported by grants UNAM-DGAPA-PAPIIT IN104717, IN108715 and CONACyT CB-2013/219840.}}
\affiliation{
  Instituto de Matem{\'a}ticas, UNAM, M{\'e}xico}
\keywords{directed graph, acyclic coloring, DFS algorithm}
\received{2018-10-24}
\revised{2019-04-12}
\accepted{2019-05-4}
\begin{document}
\publicationdetails{21}{2019}{1}{7}{4914}
\maketitle
\begin{abstract}
The \emph{chromatic number} of a graph $G$, denoted by $\chi(G)$, is the minimum $k$ such that $G$ admits a $k$-coloring of its vertex set in such a way that each color class is an independent set (a set of pairwise nonadjacent vertices). 

The \emph{dichromatic number} of a digraph $D$, denoted by $\chi_A(D)$, is the minimum $k$ such that $D$ admits a $k$-coloring of its vertex set in such a way that each color class is acyclic. 
 
In 1976, Bondy proved that the chromatic number of a digraph $D$ is at most its \emph{circumference}, the length of a longest cycle.

  Given a digraph $D$, we will construct three different graphs whose chromatic numbers bound $\chi_A(D)$.

  Moreover, we prove: i) for integers $k\geq 2$, $s\geq 1$ and $r_1, \ldots, r_s$ with $k\geq r_i\geq 0$ and $r_i\neq 1$ for each $i\in[s]$, that if all cycles in $D$ have length $r$ modulo $k$ for some $r\in\{r_1,\ldots,r_s\}$, then $\chi_A(D)\leq 2s+1$; ii) if $D$ has girth $g$ and there are integers $k$ and $p$, with $k\geq g-1\geq p\geq 1$ such that $D$ contains no cycle of length $r$ modulo $\lceil \frac{k}{p} \rceil p$ for each $r\in \{-p+2,\ldots,0,\ldots,p\}$, then $\chi_A (D)\leq \lceil \frac{k}{p} \rceil$; iii) if $D$ has \emph{girth} $g$, the length of a shortest cycle, and circumference $c$, then $\chi_A(D)\leq \lceil \frac{c-1}{g-1} \rceil +1$, which improves, substantially, the bound proposed by Bondy. 
  
 Our results show that if we have more information about the lengths of cycles in a digraph, then we can improve the bounds for the dichromatic number known until now.
\end{abstract}

\section{Introduction}
\label{sec:in}
Let $G=(V(G),E(G))$ an undirected graph with vertex set $V(G)$ and edge set $E(G)$. A set $I\subset V$ is called \emph{independent} whenever $I$ has no pair of adjacent vertices. 
A \emph{proper $k$-coloring} of $G$ is a function $c\colon V(G)\to \{1,2, \ldots,k\}$ such that each color class, $c^{-1}(i)$, is an independent set for each $i \in \{1,2, \ldots,k\}$. We say that $G$ is \emph{$k$-colorable} if there exists a proper $k$-coloring of $G$.
The \emph{chromatic number} of $G$, denoted by $\chi(G)$, is the minimum $k$ such that $G$ admits a proper $k$-coloring of its vertex set.

Let $D =(V(D),A(D))$ be a loopless directed graph (digraph) with vertex set $V(D)$ and arc set $A(D)$.
A digraph $H$ is a \emph{subdigraph} of $D$ if $V(H)\subset V(D)$ and $A(H)\subset A(D)$ and every arc in $A(H)$ has both end vertices in $V(H)$; if $V(H)=V(D)$ we sill say that $H$ is a \emph{spanning} subdigraph of $D$. The \emph{subdigraph induced} by $S\subset V(D)$, denoted by $D\langle U \rangle$, is the subdigraph of $D$ whose vertex set is $U$ and whose arc set consists of all arcs in $D$ that have both end vertices in $U$. 
A digraph $D$ is \emph{strong} if, for every pair $u$, $v$ of distinct vertices in $V(D)$, there exists a directed path from $u$ to $v$. A \emph{strong component} of a digraph $D$ is a maximal induced subdigraph of $D$ which is strong.
A digraph $D$ is said to be \emph{acyclic} if $D$ contains no directed cycle; and a set of vertices $U$ of a digraph $D$ is \emph{acyclic} whenever $D\langle U\rangle$ is an acyclic digraph.
The dichromatic number of a digraph $D$, denoted by $\chi_A(D)$, is the minimum $k$ such that $D$ admits a $k$-coloring of its vertex set in such a way that each color class is acyclic. This number was independently introduced by Mohar and his collaborators in two papers, \cite{Bokal2004227} and \cite{Mohar2003107}, and by \cite{Neumann-Lara1982265}, as a generalization of the chromatic number of graphs, in this way: take a loopless digraph $D$ such that all of its arcs are symmetric and consider the graph $G$ obtained from $D$ by replacing each pair of symmetric arcs by an undirected edge, then the dichromatic number of $D$ is equal to the chromatic number of $G$, as acyclic colorings in $D$ are proper colorings in $G$ and vice versa. 

As the dichromatic number of a digraph is a generalization of the chromatic number of an undirected graph, many authors have studied the way to generalize results for proper colorings in graphs to acyclic colorings in digraphs, as \cite{Bokal2004227, Harutyunyan2011170, Harutyunyan20121823, Keevash2013693, Mohar2003107}. In particular, \cite{Chen2007923} proved  that the problem of determining if the dichromatic number of a given digraph $D$ equals 2 is $NP$-complete, even when $D$ is a tournament.

\cite{Bondy1976277} proved that the chromatic number of a digraph $D$ is at most its \emph{circumference}, the length of a longest cycle. Notice that any proper coloring of a loopless digraph $D$ gives adjacent vertices different colors, so it is also an acyclic coloring and thus the chromatic number is an upper bound for the dichromatic number of $D$. In this way, the circumference is also an upper bound for $\chi_A(D)$.

\cite{Neumann-Lara1982265} proved  that 
\begin{math}\chi_A(D)=  \max\{ \chi_A(H)\ |\ H \text{ is a strong component of } D\}, \end{math} for any digraph $D$.
Thanks to this, we can focus our attention only on strongly connected digraphs.

Concerning acyclic colorings,  \cite{Neumann-Lara1982265} proved that for any fixed integers $k$ and $r$ with $ k\geq r\geq 2$, 
whenever a digraph $D$ contains no cycle of length 0 or 1 modulo $r$, then 
\begin{math} \chi_A(D)\leq k. \end{math}
In this direction, \cite{Chen2015210} gave a more general result. They proved for integers $k$ and $r$ with $k\geq 2$ and 
\begin{math} k\geq r\geq 1 \end{math} 
that: i) if a digraph $D$ contains no cycle of length $1$ modulo $k$, then $D$ can be colored with $k$ colors so that each color class is a stable set; ii) if a digraph $D$ contains no cycle of length $r$ modulo $k$, then $D$ can be colored with $k$ colors so that each color class induces an acyclic subdigraph of $D$; iii) if an undirected graph $G$ contains no cycle of length $r$ modulo $k$, then $G$ is $k$-colorable if $r\neq 2$ and $(k+1)$-colorable otherwise.

\cite{Tuza1992236} used a depth first search (DFS) tree of a graph to study its proper colorings, he proved that if a graph $G$ has an orientation $D$ such that every cycle $C$ of length 
\begin{math} l(C)\equiv 1 \pmod{k} \end{math} 
contains at least $l(C)/k$ arcs in each direction, then 
\begin{math} \chi(G)\leq k; \end{math}
and as a consequence he obtained that if, for some positive integer $k$ ($k\geq 2$), an undirected graph $G$ contains no cycle whose length minus one is a multiple of $k$, then $G$ is $k$-colorable. Moreover, by using the DFS algorithm, if $G$ has $n$ vertices and $m$ edges, then a proper $k$-coloring can be found in $c(n+m)$ steps, for some constant $c$ independent of the value of $k$.
Now we will use the directed version of the DFS algorithm to study the dichromatic number of a digraph.

We will be quoting the following results on the chromatic number of graphs:

\begin{theorem}[\cite{Brooks1941194}]
\label{brooks' theorem}
Let $G$ be a connected graph with maximum degree $\Delta$. If $G$ is neither an odd cycle nor a complete graph, then 
\begin{math} \chi(G)\leq \Delta. \end{math}
Otherwise 
\begin{math} \chi(G)=\Delta + 1. \end{math}
\end{theorem}

\begin{theorem}[\cite{Tuza1992236}]
\label{tuza}
Let $G$ be a graph; $k$ be an integer such that $k \geq 2$. If there is no cycle of length $1$ modulo $k$ in $G$, then 
\begin{math} \chi_A(G)\leq k. \end{math}
\end{theorem}

A digraph $T$ is an \emph{tree} if $T$ is connected, it contains no directed cycle and its underlying graph contains no cycle. 
Moreover, $T$ is an \emph{out-tree} if $T$ is a tree with just one vertex $r$ of in-degree zero and any other vertex different from $r$ has in-degree one. The vertex $r$ is the \emph{root} of $T$. 
If an out-tree $T$ is a spanning subdigraph of a digraph $D$, $T$ is called an
\emph{out-branching}.

A fixed arbitrary tree $T$ obtained from $D$ by the DFS algorithm will be called a \emph{DFS tree} of $D$, which is an out-branching rooted at a vertex $r$. 

In this paper we prove several properties of a DFS tree $T$ of a strongly connected digraph $D$, which allow us to partition $V(D)$ into acyclic sets and in this way we obtain upper bounds for $\chi_A(D)$. 
Let $T$ be a DFS tree of a strongly connected digraph $D$ with girth $g$, and let $V_0$, $V_1$, \ldots, $V_t$ be the levels of $T$: 
\begin{enumerate}
\item Let $G_T$ be the undirected graph with vertex set 
\begin{math} V(G_T)=V(T)=V(D) \end{math} 
and $uv\in E(G_T)$ whenever there is a backward arc between $u$ and $v$ in $D$. Then  
\begin{math} \chi_A(D)\leq \chi (G_T). \end{math}
\item Let $G^D$ be the undirected graph with vertex set 
\begin{math} V(D)=\{V_0, V_1, \ldots, V_t\} \end{math} 
and such that $V_iV_j\in E(G^D)$ where $i>j$ if and only if there is a backward arc $(u,v)\in A(D)$ with $u\in V_i$ and $v\in V_j$. Then 
\begin{math} \chi_A(D)\leq \chi (G^D). \end{math}
\item Let 
\begin{math} k=\lceil{\frac{t+1}{g-1}}\rceil; \end{math} 
\begin{math} U_h = \bigcup_{j=0}^{g-2} V_{h(g-1)+j} \end{math} 
for each $h\in [k]_0$, where 
\begin{math} V_{(k-1)(g-1)+j} = \emptyset \end{math} 
whenever 
\begin{math} (k-1)(g-1)+j>t; \end{math}
and $G$ be the undirected graph with vertex set 
\begin{math} V(G)=\{U_0, \ldots, U_{k-1}\} \end{math} 
and $U_iU_j\in E(G)$ whenever $i>j$ and there is a backward arc $(u,v)\in A(D)$ such that $u\in U_i$ and $v\in U_j$. Then 
\begin{math} \chi_A(D)\leq \chi (G). \end{math}
\end{enumerate}

We also prove our main results:

\begin{theorem}
\label{theo cycles of length r_i modulo k}
Let $k$ and $s$ be two integers such that $s\geq 1$; let $r_1$, \ldots, $r_s$ be $s$ integers such that 
\begin{math} k\geq r_i\geq 0 \end{math} 
for each $i\in[s]$ and $D$ a strongly connected digraph such that all cycles in $D$ have length $r$ modulo $k$ for some 
\begin{math} r\in\{r_1,\ldots,r_s\}. \end{math}
If $r_i\neq 1$ for each $i\in[s]$, then 
\begin{math} \chi_A(D)\leq 2s+1. \end{math}
\end{theorem}

\begin{theorem}
\label{theo 03}
Let $D$ be a strongly connected digraph with girth $g$; $k$ and $p$ be two integers such that 
\begin{math} k\geq g-1 \geq p \geq 1. \end{math}
If there is no cycle of length $r$ modulo 
\begin{math} \hat{k}=\lceil{\frac{k}{p}}\rceil{p} \end{math} 
in $D$ for each 
\begin{math} r\in \{-p+2,\ldots, 0,\ldots, p\}, \end{math}
then 
\begin{math} \chi_A(D)\leq \lceil{\frac{k}{p}}\rceil. \end{math}
\end{theorem}

\begin{theorem}
\label{circumference and girth}
Let $D$ be a digraph with at least one cycle and let $c$ and  $g$ be its circumference and girth, respectively. Then 
\begin{math} \chi_A(D)\leq \lceil{ \frac{c-1}{g-1}}\rceil+1. \end{math}
\end{theorem}

Our results show that if we have more information about the lengths of cycles in a digraph, then we can improve the bounds for the dichromatic number known until now. The restriction on the cycles lengths in Theorems \ref{theo cycles of length r_i modulo k} and \ref{theo 03} provide better bounds. In Theorem \ref{theo cycles of length r_i modulo k} whenever $s<\lfloor \frac{k}{2} \rfloor$, $\chi_A(D)<k$ and in Theorem \ref{theo 03} whenever $k=sp$ with $s\geq 3$, $\chi_A(D)\leq s < k$. Our results are different from those obtained by \cite{Chen2015210}. However, these two cases lower their bound. 

Notice that along our proofs we will use the DFS algorithm to obtain acyclic colorings and this algorithm is polynomial. So, we will be able to color digraphs in an ordered fast way. 

The text is divided into four sections, including the introduction. In Section \ref{sec:definitions} we will give the definitions and the notation that we will use throughout this document. In section \ref{sec:preliminary results} we will see some properties of a digraph $D$ derived from the structure of a DFS tree of $D$ and we will construct three different graphs from $D$ whose chromatic numbers will bound the dichromatic number of $D$. Finally, in Section \ref{sec:main results} we will prove our three main results: Theorems \ref{theo cycles of length r_i modulo k}, \ref{theo 03} and \ref{circumference and girth}. 

% You may scarsely use \clearpage to advance to a new page if this
% improves the readability of the document structure
%\clearpage
\section{Definitions}
\label{sec:definitions}
In this paper $D=(V(D),A(D))$ will denote a loopless digraph. Two vertices $u$ and $v$ are adjacent if $(u,v)$ or $(v,u)$ is in $A(D)$. The \emph{underlying graph} of $D$ is the graph $G$ obtained from $D$ by replacing each arc $(u,v)$ for the undirected edge $uv$ and then deleting all multiple edges between every pair of vertices apart from one.
The subdigraph induced by a set of vertices $U\subseteq V(D)$ will be denoted by 
\begin{math} D\langle U\rangle; \end{math}
and if $H$ is a subdigraph of $D$, the subdigraph induced by $V(H)$ will be simply denoted by 
\begin{math} D\langle H\rangle. \end{math} 

The set of integers $\{1,2, \ldots, k\}$ will be denoted by $[k]$, and  $[k]_0$ will denote the set $\{0$, 1, \ldots, $k-1\}$.

Here the paths and cycles are always directed.
The \emph{girth} (resp. \emph{circumference}) of $D$ is the shortest (resp. longest) cycle. 
A \emph{proper $k$-coloring} of $D$ is a proper $k$-coloring of its underlying graph $G$ (this is a coloring of $V(D)$ such that no two adjacent vertices are colored alike). The \emph{chromatic number} of $D$, denoted by $\chi(D)$, is the minimum $k$ such that $D$ admits a proper $k$-coloring.
An \emph{acyclic $k$-coloring} of $D$ is a function 
\begin{math} c\colon V(D)\to [k] \end{math} 
such that $c^{-1}(i)$ induces an acyclic subdigraph in $D$ for each $i\in [k]$. The \emph{dichromatic number} of $D$, denoted by $\chi_A(D)$, is the minimum $k$ such that $D$ admits an acyclic $k$-coloring. 

A subdigraph $T$ of a digraph $D$ is an \emph{out-branching} if: (i) $T$ is connected, it contains no directed cycle and its underlying graph contains no cycle, (ii) $T$ has just one vertex $r$ of in-degree zero and any other vertex different from $r$ has in-degree one, the vertex $r$ is the \emph{root} of $T$ and (iii) $T$ is a spanning subdigraph of $D$ ($V(T)=V(D)$).

A fixed arbitrary tree $T$ obtained from $D$ by the DFS algorithm will be called a \emph{DFS tree} of $D$, which is an out-branching rooted at a vertex $r$. 

Recall that: $f(v)$ denotes the time a vertex $v\in V(T)$ is explored by the algorithm and, is called the \emph{DFS label} of $v$; each vertex has a unique DFS label and any two different vertices have different DFS labels;
%We call a \emph{DFS forest} to a subdigraph $F$ with $V(D)=V(F)$ obtainend by the DFS algorithm, a tree $T$ in $F$ is called a \emph{DFS tree} and the only vertex $r$ in $T$ whose in-degree is zero is the \emph{root} of $T$.  If $D$ is a strongly connected digraph, then the DFS algorithm produces a tree $T$ such that $V(D)=V(T)$. 
%Let $D$ be a strongly connected digraph. Consider $T$ a DFS tree rooted in $r$. \\
* Whenever there is a $uv$-path in $T$, 
$v$ is a \emph{descendant} of $u$ and $u$ is an \emph{ancestor} of $v$.
Finally, an arc $(u,v)\in A(D)$ is: (i) a  \emph{tree arc} if $(u,v)\in A(T)$, (ii) a  \emph{backward arc} if $u$ is a descendant of $v$, (iii) a
\emph{forward arc} if $u$  is an ancestor of $v$ and (iv) a \emph{cross arc} if $u$ is neither an ancestor nor a descendant of $v$.
For further details we refer the reader to the book of \cite{Bang-Jensen2009}, pages 26-29.

Along this paper, $D$ will be a strongly connected digraph and $T$ will be a DFS tree of $D$ rooted at a vertex $r$ and $V$ will denote $V(D)=V(T)$.

\section{Preliminary results}
\label{sec:preliminary results}

In this section we will see some properties of a digraph $D$ derived from the structure of a DFS tree of $D$ that will be useful in the proofs of the main results; and we will construct three different graphs from $D$ whose chromatic numbers bound the dichromatic number of $D$.

We will start with four properties of a DFS tree of a digraph derived from the definitions. Later we will prove Lemma \ref{lemma backward arc}, that will be the key to construct acyclic sets of vertices in a digraph. After we will see different partitions of the vertex set of a digraph $D$ into acyclic sets that are based on a DFS tree of $D$ and we also see a few (simpler) bounds for the dichromatic number of $D$.

\begin{remark}
\label{DFS 0} It follows from the definitions that in $D$ there are four kinds of arcs: tree arcs, forward arcs, backward arcs and cross arcs. If $(u,v)$ is: (i) a tree arc then $f(u)<f(v)$, (ii) a forward arc then $f(u)<f(v)$, (iii) a backward arc then $f(u)>f(v)$ and (iv) a cross arc then $f(u)>f(v)$ (see Proposition \ref{cross arcs}).

Observe that, as there are only four kinds of arcs, if an arc  $(u,v)\in A(D)$ satisfies $f(u)<f(v)$, then it is either a tree arc or a forward arc; and if $f(u)>f(v)$, then it is either a backward arc or a cross arc.
\end{remark}

\begin{remark}
\label{DFS 1} If $f(u)<f(v)$ and $v$ is an out-neighbor of $u$ in $D$, then there is a $uv$-path in $T$. 
Even more, if $f(u)<f(v)$ and $v$ is a descendant of $u$, then for each $w$ such that 
\begin{math} f(u)<f(w)<f(v) \end{math} 
we have $w$ is also a descendant of $u$. This follows from the definition of the DFS algorithm, as all vertices explored between the time a vertex $u$ is visited for the first time, namely $f(u)$, and before $u$ is completely processed are descendants of $u$ (book of \cite{Gross2006}, page 524).
\end{remark}

\begin{remark}
\label{DFS 2}
Whenever there is a $uv$-path in $T$, this is unique in $T$, $u$ must be an ancestor of $v$ in $T$ and $f(u)<f(v)$; moreover, every walk in $T$ is a path. 
\end{remark}

Next proposition was proved in the book of \cite{Gross2006}, page 525, and it stands that each cross arc goes from a branch explored later to a branch explored earlier.

\begin{proposition}
\label{cross arcs}
Let $e$ be a cross arc of a DFS tree from a vertex $u$ to a vertex $v$. Then $f(u)>f(v)$.
\end{proposition}

\begin{lemma}
\label{lemma backward arc}
Each cycle in $D$ contains at least one backward arc. 
\end{lemma}
\begin{proof}
Let 
\begin{math} C=(u_0,u_1,\ldots,u_{k-1},u_0) \end{math} 
be a cycle in $D$. 
In what follows the subscripts will be taken modulo $k$.

Consider the set of DFS labels of vertices in $C$, this set consists of $k$ different natural numbers, so we can chose the greatest. Take $M\in [k]_0$ such that $f(u_M) > f(u_j)$ for each $j\neq M$, $j\in [k]_0$. Hence, $f(u_{M-1})< f(u_M)$ and $(u_{M-1},u_M)$ is a tree or a forward arc (see Remark \ref{DFS 0}). 

Notice that there is an index $s\in [k]_0$ such that $f(u_{s-1})>f(u_s)$, as $f(u_M)> f(u_{M+1})$. Let $t$ be the first such index starting at $M$ and going on the reverse order of $C$.
Thus, $f(u_{t-1})>f(u_t)$ and 
\begin{math} f(u_t)<f(u_{t+1})<\cdots <f(u_{M-1})<f(u_M). \end{math} 

Observe that $u_t$ is an ancestor of $u_M$ as $(u_t,u_{t+1})$, \ldots, $(u_{M-1},u_M)$ are arcs in $D$ with $f(u_i)<f(u_{i+1})$ for each $i=t$, $t+1$,\ldots, $M-1$, and thus they are forward or tree arcs, by Remark \ref{DFS 0}.

\begin{claim} $t\neq M$. This is clear as $f(u_{M-1})<f(u_M)$ and $f(u_{t-1})>f(u_t)$ by the definitions of $M$ and $t$. 
\end{claim}

\begin{enumerate}[{Case }1:]
\item $t=M+1$. Then $u_t=u_{M+1}$ is an ancestor of $u_M$, and thus $(u_M,u_{M+1})$ is a backward arc (see * in the definitions).
\item $t\neq M+1$ ($t-1\neq M$). Then $u_{t-1}$, $u_t$ and $u_M$ are three different vertices which satisfy 
\begin{math} f(u_t)< f(u_{t-1})<f(u_M) \end{math} 
and $u_t$ is an ancestor of $u_M$, then by Remark \ref{DFS 1}, $u_{t}$ is an ancestor of $u_{t-1}$ and thus $(u_{t-1},u_{t})\in A(C)$ is a backward arc. 
\end{enumerate}
\end{proof}

\begin{definition}
A cycle $C$ in $D$ is an \emph{elementary cycle} whenever there is a backward arc $f=(u,v)\in A(C)$ such that $C-f$ is a $vu$-path contained in $T$.
\end{definition}

A direct consequence of Lemma \ref{lemma backward arc} is that if $c$ is a coloring of a strongly connected digraph $D$, such that end vertices of backward arcs (with respect to some DFS tree) have different colors, then $c$ is an acyclic coloring of $D$. This is clear as all cycles of $D$ contain at least one backward arc, and no backward arc is monochromatic. 

We will take advantage of this fact to construct bounds for the dichromatic number of $D$.

\begin{definition} 
\label{def:underlying backward graph}
Given  a connected subdigraph $H$ of $T$ define the undirected graph $G_H$ with vertex set $V(G_H)=V(H)$ and $uv\in E(G_H)$ whenever there is a backward arc between $u$ and $v$ in $D$. We will call $G_H$ the \emph{underlying backward graph} of $H$ relative to $T$.
\end{definition}

\begin{lemma}
\label{lemma underlying back graph}
Let $D$ be a strongly connected digraph and $T$ a DFS tree rooted at $r$. Let $H$ be a connected subdigraph of $T$ and  $G_H$ its underlying backward graph relative to $T$. Then  
\begin{math} \chi_A(D\langle H \rangle)\leq \chi (G_H). \end{math}
\end{lemma}
\begin{proof}
Let 
\begin{math} c\colon V(G_H)\to [\chi(G_H)] \end{math} 
be a proper coloring of $V(G_H)$ with $\chi (G_H)$ colors. As $c$ is proper, each color class is a stable set. 
This coloring in $V(G_H)=V(H)$ is a coloring of $D\langle H \rangle$ in which the ends of each backward arc have different colors.

By Lemma \ref{lemma backward arc}, any cycle in  $D$ and in particular in $D\langle H \rangle$ contains a backward arc, therefore  $D\langle H \rangle$ has no monochromatic cycle with this coloring. Hence 
\begin{math} \chi_A(D\langle H \rangle)\leq \chi (G_H). \end{math}
\end{proof}

\begin{corollary}
Let $D$ be a strongly connected digraph and $T$ a DFS tree rooted at $r$; and $G_T$ its underlying backward graph relative to $T$. Then  
\begin{math} \chi_A(D) \leq \chi (G_T). \end{math}
\end{corollary}

\begin{proposition}
\label{proposition chromatic number of a branch}
Let $D$ be a strongly connected digraph and let $T$ be a DFS spanning tree of $D$ with root $r$. Let $x_1$,  \ldots, $x_k$ be the out-neighbors of $r$ in $T$. If $T_1$,  \ldots, $T_k$ are the subtrees of $T$ such that 
$x_i\in V(T_i)$, 
\begin{math} T=\bigcup_{i=1}^k T_i\end{math} 
and 
\begin{math} V(T_i) \cap V(T_j)=\{r\} \end{math} 
for all $i\neq j$. Then 
\begin{math} \chi_A(D)\leq \max_{i\in [k]}\{\chi(G_{T_i})\} \end{math} 
for all $i\in [k]$.
\end{proposition}
\begin{proof}
Take 
\begin{math} d=\max_{i\in[k]}\{\chi(G_{T_i})\}. \end{math} 
Since 
\begin{math} \chi(G_{T_i})\leq d \end{math}
for each $i\in [k]$, let 
\begin{math} c_i\colon V(G_{T_i})\to [d] \end{math} 
be a proper coloring for each $i \in [k]$. 
We may assume that $c_i(r)=1$ for each $i\in [k]$, since we can permute two colors to get this situation.

Let's define a color function 
\begin{math} c\colon V \to [d+1] \end{math} 
on $D$ as follows:
$c(v)=c_i(v)$ for all $v\in V(G_{T_i})$. Notice that 
\begin{math} V=\bigcup_{i=1}^k V(G_{T_i})=\bigcup_{i=1}^k V(T_i), \end{math} 
\begin{math} V(T_i)\cap V(T_j)=\{r\} \end{math} 
whenever $i\neq j$ and $c_i(r)=1$ for each $i\in [k]$, then $c$ is well defined all over $D$.

Let $U_1$, \ldots, $U_d$ be the color classes in $D$. Proceeding by contradiction we will prove that $D\langle U_s\rangle$ is acyclic for all $1\leq s \leq d$. 

Suppose that there is a cycle $C$ in $D\langle U_s \rangle$ for some $s\in [k]$. By Lemma \ref{lemma backward arc}, there is a backward arc $(u,v)\in A(C)$ and $c(u)=c(v)$. By definition of backward arc, $u$ is a descendant of $v$ and thus there is an $i\in [k]$ such that 
\begin{math} u,v\in V(T_i)=V(G_{T_i}) \end{math} 
and $uv\in E(G_{T_i})$. Therefore, 
\begin{math} c_i(u)=c(u)=c(v)=c_i(v), \end{math} 
which is impossible as $c_i$ is a proper coloring of $G_{T_i}$. Hence, 
$D\langle U_s \rangle$ is acyclic for each $s\in [k]$ and $\chi_A(D)\leq d$.
\end{proof}

\begin{notation}
Given a strongly connected digraph $D$ and a DFS tree  $T$ of $D$ rooted at $r$. We will denote by $P_u$  the unique $ru$-path contained in $T$  for each vertex $u\in V$.
\end{notation}

\begin{lemma}
\label{lemma paths between relatives}
Let $D$ be a strongly connected digraph and $T$ be a DFS tree of $D$ rooted at $r$.  If $(u,v)$ is a backward arc, then there is a (unique)  $vu$-path $P$ in $T$ such that 
\begin{math} P_u=P_v\cup P \end{math} 
and 
\begin{math} l(P_u)=l(P_v)+l(P). \end{math}
\end{lemma}
\begin{proof}
Let $(u,v)$ such an arc, so $u$ is a descendant of $v$; by * in the definition of the DFS tree, there is a $vu$-path $P$ in $T$. By Remark \ref{DFS 2}, $P$ is unique. 

Now, $P_v\cup P$ is a $ru$-walk which, by Remark \ref{DFS 2}, is a path and 
\begin{math} P_u=P_v\cup P \end{math} 
with 
\begin{math} l(P_u)=l(P_v)+l(P). \end{math} 
\end{proof}

\begin{lemma}
\label{lemma partition of V}
Let $D$ be a strongly connected digraph and $T$ be a DFS tree of $D$ rooted at $r$. Let 
\begin{math} V_i=\{u\in V\ |\ l(P_u)=i\} \end{math} 
for each $i$, 
\begin{math} 0\leq i \leq t, \end{math}
where $t$ is the length of a longest path in $T$. The following assertions hold: (i) 
\begin{math} V=\bigcup_{i=0}^t V_i, \end{math} 
(ii) 
\begin{math} V_i\cap V_j=\emptyset \end{math} 
for $i\neq j$ and (iii) 
\begin{math} D\langle V_i \rangle \end{math} 
is acyclic.
\end{lemma}
\begin{proof}
Assertions (i) and (ii) follow directly from the definitions of a DFS tree and the sets $V_i$'s. 

In order to prove (iii) assume by contradiction that 
\begin{math} D\langle V_i\rangle \end{math}
has a cycle 
\begin{math}C=(u_0,u_1,\ldots,u_{k-1},u_0).\end{math} 
By Lemma \ref{lemma backward arc}, $C$ has at least one backward arc $(u_, u^+)$ (where $u\in V(C)$ and $u^+$ is its successor in $C$) and, by Lemma \ref{lemma paths between relatives}, there is a $u^+u$-path $P$ in $T$, such that 
\begin{math} P_u=P_{u^+}\cup P \end{math} 
and 
\begin{math} l(P_u)=l(P_{u^+})+l(P)>l(P_{u^+}). \end{math} 
Then 
\begin{math} i=l(P_u)>l(P_u^+)=i, \end{math} 
which is a contradiction.   
\end{proof}

\begin{definition}
The set 
\begin{math} V_i=\{u\in V\ |\ l(P_u)=i\} \end{math} 
will be called the \emph{$i$-th generation} or the \emph{$i$-th level} of $T$, for each $i\in [t+1]_0$, where $t$ is the length of a longest path in $T$; and we will say that $T$ has \emph{length} $t$.
\end{definition}

\begin{remark}
\label{remark backward arc}
If $(u,v)$ is a backward arc, then $u\in V_i$ and $v\in V_j$, with $i>j$. This is a direct consequence of Lemma \ref{lemma paths between relatives}. 
\end{remark}

\begin{definition}
\label{def:backward spine}
Let $D$ be a strongly connected digraph $D$ and $T$ a DFS tree of $D$ rooted at $r$ such that $T$ has length $t\geq 1$ and let $V_0$, \ldots, $V_t$ be its generations. Let $G^D$ be the graph with vertex set 
\begin{math} V(G^D)=\{V_0,\ldots,V_t\} \end{math} 
and such that $V_iV_j \in E(G^D)$ where $i>j$ if and only if there is a backward arc $(u,v)\in A(D)$ with $u\in V_i$ and $v\in V_j$. $G^D$ will be called the \emph{backward spine} of $D$ relative to $T$.
\end{definition}

The chromatic number of $G^D$ is an upper bound for the dichromatic number of $D$ as we will see in the following proposition:

\begin{proposition}
\label{proposition backward spine}
Let $D$ be a strongly connected digraph, $T$ a DFS tree rooted at $r$ and $G^D$ its backward spine relative to $T$. If 
\begin{math} c'\colon \{V_0,\ldots,V_t\}\to [\chi(G^D)] \end{math} 
is a proper coloring of $G^D$, then the mapping 
\begin{math} c\colon V\to [\chi(G^D)], \end{math}
defined as $c(v)=c'(V_i)$ whenever $v\in V_i$, is an acyclic coloring of $D$. 
Moreover, 
\begin{math}\chi_A(D)\leq \chi(G^D). \end{math}
\end{proposition}
\begin{proof}
Let $d=\chi(G^D)$, 
\begin{math} c'\colon V(G^D)\to [d] \end{math} 
be a proper coloring of $G^D$ and 
\begin{math} c\colon V \to [d] \end{math} 
defined as in the hypothesis. 

First, $c$ is well defined since $V_0$, \ldots, $V_t$ form a partition of $V(D)$ and thus $c$ assigns a unique color to each vertex $v\in V(D)$.  

Second, $c$ is acyclic. Take $C$ an arbitrary cycle in $D$, by Lemma \ref{lemma backward arc}, $C$ contains a backward arc $(u,v)$. Let $i,j\in [t+1]_0$ such that $u\in V_i$ and $v\in V_j$, where $i>j$ by Remark \ref{remark backward arc}. Then, $V_i\neq V_j$, $V_iV_j\in E(G^D)$ and 
\begin{math} c'(V_i)\neq c'(V_j). \end{math}
Hence, 
\begin{math} c(u)=c'(V_i)\neq c'(V_j)=c(v), \end{math} 
and $C$ is not monochromatic in $D$. Since there is no monochromatic cycle in $D$, the color class $c^{-}(i)$ is acyclic for each $i\in [d]$.
Moreover, $c$ is an acyclic coloring of $D$ with $d$ colors, therefore 
\begin{math} \chi_A(D)\leq d. \end{math}
\end{proof}

\begin{corollary}
Let $D$ be a strongly connected digraph and $T$ a DFS tree rooted at $r$ such that $T$ has length $t$. Then 
\begin{math} \chi_A(D)\leq t+1. \end{math}
\end{corollary}
\begin{proof}
Let $G^D$ the backward spine of $D$ relative to $T$, then 
\begin{math} |V(G^D)|=t+1 \end{math} 
and 
\begin{math} \Delta(G^D)\leq t, \end{math}
where $\Delta(G^D)$ denotes the maximum degree in $G^D$. Therefore, by Theorem \ref{brooks' theorem} and Proposition \ref{proposition backward spine}, 
\begin{math} \chi_A(D)\leq \chi(G^D)\leq t+1. \end{math}

Moreover, if $G^D$ is neither an odd cycle nor a complete graph, then 
\begin{math} \chi_A(D)\leq \chi(G^D)\leq t. \end{math}
\end{proof}

If we have some knowledge about the lengths of cycles in our digraph $D$, we may find a bound for $\chi_A(D)$ by means of Lemma \ref{lemma partition of V}, as we will see. 
\cite{Chen2015210} proved for two integers $r$ and $k$ with $k\geq 2$ and  $k\geq r\geq 1$ that if a digraph $D$ contains no cycle of length $r$ modulo $k$, then $\chi_A(D)\leq k$. 
We will give other bounds for $\chi_A(D)$ in terms of cycle lengths.

\begin{theorem}[\cite{Chen2015210}]
\label{dichro theo 01}
Let $r$ and $k$ be two integers with $k\geq 2$ and  $k\geq r\geq 1$. If a digraph $D$ contains no cycle of length $r$ modulo $k$, then 
$\chi_A(D)\leq k$. 
\end{theorem}

\begin{corollary}
\label{corollary dichro theo 01}
Let $D$ be a digraph with girth $g$ and circumference $c$ with 
\begin{math} g-1\leq c-g+2 \end{math} 
\begin{math} (g\leq \frac{c+3}{2}). \end{math}
Then 
\begin{math} \chi_A(D) \leq c-g+2. \end{math}
\end{corollary}
\begin{proof}
Note that $D$ has no cycle of length $g-1$ modulo $c-g+2$. Suppose, by  contradiction, that $C$ is a cycle in $D$ such that 
\begin{math} l(C)\equiv g-1 \pmod{c-g+2} \end{math} 
then 
\begin{math} l(C) = g-1 + q(c-g+2) \end{math} 
for some integer $q\geq 0$. If $q=0$ then  $l(C)=g-1$ and this is impossible for a digraph of girth $g$. Therefore, $q\geq 1$ and so 
\begin{math} l(C)= g-1 + q(c-g+2)\geq g-1 + c-g+2=c+1, \end{math} 
which is a contradiction.

Now, since $D$ has no cycle of length $g-1$ modulo $c-g+2$, by Theorem \ref{dichro theo 01} we have 
\begin{math} \chi_A(D)\leq c-g+2. \end{math}
\end{proof}

\begin{remark}
The bound proposed in corollary \ref{corollary dichro theo 01} is sharp.
For any digraph $D$ with at least one cycle 
\begin{math} \chi_A(D)\geq 2. \end{math} 
Now, take $D$ such that all its cycles have length 3, then 
\begin{math} \chi_A(D) \leq 2. \end{math} 
Thus, the bound is reached. 
\end{remark}

Now we will see a bound for $\chi_A(D)$ when all cycles in $D$ have length $r$ modulo $k$. First, let us prove next Lemma.
which is the directed version of Theorem \ref{tuza} and a particular case of Theorem \ref{dichro theo 01}.  

\begin{lemma}
\label{lemma no cycles of length 1 modulo k}
Let $k$ be an integer with $k\geq 2$. Let $D$ be a strongly connected digraph with no cycle of length $1$ modulo $k$; $T$ a DFS tree of length $t$ of $D$ rooted at $r$ and $V_0$, \ldots, $V_t$ its generations. The sets 
\begin{math} U_s=\bigcup_{i\equiv s\pmod{k}} V_i \end{math} 
for each $s\in [k]_0$ form a partition of $V$ into acyclic sets. 
\end{lemma}
\begin{proof}
Take $U_s$ as in the hypothesis for each $s\in [k]_0$. 

First, 
\begin{math} \bigcup_{s=0}^{k-1} U_s = \bigcup_{i=0}^t V_i=V \end{math} 
and 
\begin{math} U_i\cap U_j=\emptyset \end{math} 
if $i\neq j$ where $i,j\in [k]_0$, by definition.

Second, proceeding by contradiction, assume that $U_s$ is not acyclic for some $s\in[k]_0$, then there is a cycle $C$ in $D$ such that $V(C)\subset U_s$. By Lemma \ref{lemma backward arc}, there is a backward arc $(u,v)\in A(C)$. As $u,v\in U_s$, we have 
\begin{math} l(P_u)=i\equiv s \pmod{k} \end{math} 
and 
\begin{math} l(P_v)=j\equiv s \pmod{k}, \end{math} 
by definition of the generations and the $U_s$' and $i>j$, by Remark \ref{remark backward arc}. Hence, by Lemma \ref{lemma paths between relatives}, there is a $vu$-path in $T$, namely $P$, with 
\begin{math} l(P)= i-j \equiv 0 \pmod{k} \end{math} 
and thus, the elementary cycle, $P\cup(u,v)$, has length 1 modulo $k$, a contradiction.
\end{proof}

In the following Lemma we will see that $g-1$ of consecutive levels in a DFS tree induce an acyclic digraph, where $g$ is the girth of $D$. This will be useful to construct acyclic colorings.

\begin{lemma}
\label{lemma g-1 consecutive generations acyclic}
Let $D$ be a strongly connected digraph with girth $g$; $T$ a DFS tree rooted at $r$ such that $T$ has length $t$ and $V_0$, \ldots, $V_t$ its generations. Then $g-1$ consecutive generations of $T$, $V_h$, $V_{h+1}$, \ldots, $V_{h+g-2}$, induce an acyclic subdigraph in $D$, namely 
\begin{math} D\langle \bigcup_{i=0}^{g-2} V_{h+i}\rangle, \end{math}
for each $0\leq h\leq t-g+2$. 
Moreover, there is no backward arc with both ends in 
\begin{math} \bigcup_{i=0}^{g-2} V_{h+i}. \end{math}
\end{lemma}
\begin{proof}
Let $h\in [t-g+3]_0$.
Suppose by contradiction that there is a cycle $C$ in 
\begin{math} D\langle \bigcup_{i=0}^{g-2} V_{h+i} \rangle. \end{math}
By Lemma \ref{lemma backward arc}, $C$ contains a backward arc $(u,v)$. Then $u\in V_{h+i}$ and $v\in V_{h+j}$ for some 
\begin{math} 0\leq j < i \leq g-2. \end{math}
Observe that there is a $vu$-path, $P$, in $T$ such that 
\begin{math} C'=P\cup(u,v) \end{math} 
is an elementary cycle in $D$ and 
\begin{math} l(P_u)=l(P_v)+l(P), \end{math}
by Lemma \ref{lemma paths between relatives}. Hence, on the one hand \begin{math} l(P)=l(C')-1 \geq  g-1 \end{math} 
and on the other hand 
\begin{math} l(P)=l(P_u)-l(P_v)=(h+i)-(h+j)=i-j; \end{math}
therefore, $i-j\geq g-1$, which is a contradiction since 
\begin{math} 0\leq j < i \leq g-2 \end{math} 
and thus 
\begin{math} 0<i-j\leq g-2 < g-1. \end{math} 

In this way, 
\begin{math} D\langle \bigcup_{i=0}^{g-2} V_{h+i}\rangle \end{math} 
is acyclic and there is no backward arc with both ends in 
\begin{math}\bigcup_{i=0}^{g-2} V_{h+i}. \end{math}
\end{proof}

\begin{lemma}
\label{lemma partition acyclic sets}
Let $D$ be a strongly connected digraph with girth $g$, $T$ a DFS tree of $D$ rooted at $r$, length $t$ and generations $V_0$, \ldots, $V_t$; 
\begin{math} k=\lceil{\frac{t+1}{g-1}}\rceil; \end{math} 
and 
\begin{math} U_h = \bigcup_{j=0}^{g-2} V_{h(g-1)+j} \end{math} 
for each $h\in [k]_0$, where 
\begin{math} V_{(k-1)(g-1)+j}=\emptyset \end{math} 
whenever 
\begin{math} (k-1)(g-1)+j>t. \end{math}
Then $U_0$, \ldots, $U_{k-1}$ form a partition of $V$ into acyclic sets. Moreover, there is no backward arc $(u,v)$ in $D$ with both ends in the same $U_h$.  
\end{lemma}
\begin{proof}
First, as $V_0$, \ldots, $V_t$ form a partition of $V$ and 
\begin{math} t < \frac{(t+1)}{(g-1)}(g-1) \leq \lceil{\frac{t+1}{g-1}}\rceil(g-1)=k(g-1), \end{math}
it follows that each $i\in [t+1]_0$ satisfies 
\begin{math} i\leq t \leq k(g-1)-1=(k-1)(g-1)+g-2 \end{math} 
and so, by Euclidean algorithm, there exist unique $h\in[k]_0$ and $j\in[g-1]_0$ such that $i=h(g-1)+j$ and 
\begin{math} V_i\subseteq U_h \end{math} 
and by the uniqueness of $h$ we have  
\begin{math} U_h\cap U_{h'}=\emptyset \end{math} 
if $h\neq h'$.
Hence, the sets $U_0$,\ldots, $U_{k-1}$ form a partition of $V$.
%\begin{sloppypar}
Second, from Lemma \ref{lemma g-1 consecutive generations acyclic}, the induced subdigraph 
\begin{math} D\langle U_h \rangle = D\langle \bigcup_{j=0}^{g-2} V_{h(g-1)+j} \rangle \end{math} 
is acyclic for each $h\in [k]_0$.
%\end{sloppypar}
The moreover part is consequence of Lemma \ref{lemma g-1 consecutive generations acyclic}.
\end{proof}

\begin{corollary}
Let $D$ be a strongly connected digraph with girth $g$, $T$ a DFS tree rooted and $r$ and $G^D$ the backward spine of $D$. Then the sets 
\begin{math} \{ V_{h(g-1)+j}\}_{j=0}^{g-2} \end{math} 
are independent in $G^D$ for each  $h\in [k]_0$, where 
\begin{math} V_{(k-1)(g-1)+j}=\emptyset \end{math}
whenever 
\begin{math} (k-1)(g-1)+j>t. \end{math}
\end{corollary}
\begin{proof}
From Lemma \ref{lemma partition acyclic sets}, we have that the sets 
\begin{math} U_h = \bigcup_{j=0}^{g-2} V_{h(g-1)+j} \end{math} 
are acyclic for each $h\in [k]_0$, where 
\begin{math} V_{(k-1)(g-1)+j} = \emptyset \end{math} 
whenever 
\begin{math} (k-1)(g-1)+j>t; \end{math} 
and by Lemma \ref{lemma partition acyclic sets}, there is no backward arc with both ends in 
\begin{math} U_h = \bigcup_{j=0}^{g-2} V_{h(g-1)+j} \end{math} 
and thus, by the definition of backward spine, there is no edge with both ends in 
\begin{math} \{ V_{h(g-1)+j}\}_{j=0}^{g-2} \end{math} 
for each  $h\in [k]_0$. 
\end{proof}

Next proposition gives a bound for $\chi_A(D)$, that may no be very good for an arbitrary digraph, but it is better if $D$ is strong and the ratio between the length of a longest path and the girth is not too large.

\begin{proposition}
\label{proposition path/girth}
Let $D$ be a strongly connected digraph. Let $g$ be the girth of $D$ and $l$ be an integer defined as
\begin{math} l=\min_{u\in V}\{j\ |\ j \text{ is the length of a longest path starting at } u\}. \end{math} 
Then 
\begin{math} \chi_A(D)\leq \lceil{\frac{l+1}{g-1}}\rceil. \end{math}
\end{proposition}
\begin{proof}
Let $r$ be a vertex in $D$ such that  the length of a longest path starting at $r$ in $D$ equals $l$ and let $T$ be a DFS tree of $D$ rooted in $r$, then $T$ has length $t\leq l$. Let $V_0$, \ldots, $V_t$ be the generations of $T$.

Divide the vertex set $V$ in 
\begin{math} k = \lceil{\frac{l+1}{g-1}}\rceil \end{math} 
subsets as in Lemma \ref{lemma partition acyclic sets}, this is 
\begin{math} U_h = \bigcup_{j=0}^{g-2} V_{h(g-1)+j} \end{math} 
for each $h\in [k]_0$, where 
\begin{math} V_{(k-1)(g-1)+j}=\emptyset \end{math} 
whenever 
\begin{math} (k-1)(g-1)+j>t. \end{math}

Then 
\begin{math} c\colon V\to [k]_0 \end{math} 
given by $c(v)=h$ whenever $v\in U_h$ is a $k$-coloring in which each color class is acyclic, by Lemma \ref{lemma partition acyclic sets}. Therefore 
\begin{math} \chi_A(D)\leq k. \end{math}
\end{proof}

\begin{corollary}
Let $D$ be a strongly connected digraph of order $n$ and girth $g$.
Then 
\begin{math} \chi_A(D)\leq \lceil{\frac{n}{g-1}}\rceil. \end{math}
\end{corollary}
\begin{proof}
The length of a longest path in $D$ is at most $n-1$. Then 
\begin{math} \chi_a(D)\leq \lceil{\frac{n}{g-1}}\rceil, \end{math} 
by Proposition \ref{proposition path/girth}.
\end{proof}

\begin{definition}
\label{def:condensation graph}
Let $D$, $T$ and $U_0$, \ldots, $U_{k-1}$ be as in Lemma \ref{lemma partition acyclic sets}. 
Define the \emph{condensation graph} of $D$ relative to $T$ as the graph, $G$, with vertex set 
\begin{math} V(G)=\{U_0,\ldots, U_{k-1}\} \end{math} 
and $U_iU_j\in E(G)$ whenever $i>j$ and there is a backward arc $(u,v)\in A(D)$ such that $u\in U_i$ and $v\in U_j$.
\end{definition}

\begin{proposition}
\label{proposition condensed graph}
Let $D$ be a strongly connected digraph, $T$ a DFS tree of $D$ rooted at $r$ and $G$ the condensation graph of $D$ relative to $T$. Then 
\begin{math} \chi_A(D)\leq \chi(G). \end{math}
\end{proposition}
\begin{proof}
Let $d=\chi(G)$ and 
\begin{math} c\colon V(G)\to [d] \end{math} 
a proper coloring of $G$. Define the coloring 
\begin{math} c'\colon V\to [d] \end{math} 
in $D$ as follows $c'(v)=c(U_h)$ whenever $v\in U_h$.

As $U_0$, \ldots, $U_{k-1}$ is a partition of $V$, $c'$ is well defined and, by Lemma \ref{lemma partition acyclic sets}, each color class is acyclic since each backward arc $(u,v)\in A(D)$ satisfies $u\in U_i$, $v\in U_j$, $i>j$ and $U_iU_j\in E(G)$, hence 
\begin{math} c'(u)=c(U_i)\neq c(U_j)=c'(v), \end{math}
as $c$ is proper. Therefore, by Lemma \ref{lemma backward arc}, there is no monochromatic cycle in $D$.
\end{proof}

\section{Main results}
\label{sec:main results}

In this section we will prove our three main results: Theorems \ref{theo cycles of length r_i modulo k}, \ref{theo 03} and \ref{circumference and girth}. 

%\begin{theorem}
%\label{theo cycles of length r_i modulo k}
%Let $k$ and $s$ be two integers such that $s\geq 1$; let $r_1$, \ldots, $r_s$ be $s$ integers such that $k\geq r_i\geq 0$ for each $i\in[s]$ and $D$ a strongly connected digraph such that all cycles in $D$ have length $r$ modulo $k$ for some 
%\begin{math} r\in\{r_1,\ldots,r_s\}. \end{math} 
%If $r_i\neq 1$ for each $i\in[s]$, then 
%\begin{math} \chi_A(D)\leq 2s+1. \end{math}
%\end{theorem}

First we prove Theorem \ref{theo cycles of length r_i modulo k}.

\begin{proof}[of Theorem \ref{theo cycles of length r_i modulo k}]
Let $T$ be a DFS tree of $D$ rooted at $\hat{r}$ of length $t$ and $V_0$, \ldots, $V_t$ its generations. 

Since $r_i\neq 1$ for each $i\in[s]$ and all cycles in $D$ have length $r$ modulo $k$ for some 
\begin{math} r\in\{r_1\ldots,r_s\}, \end{math} 
$D$ has no cycle of length 1 modulo $k$, by Lemma \ref{lemma no cycles of length 1 modulo k} the sets 
\begin{math} U_s=\bigcup_{i\equiv s\pmod{k}} V_i \end{math} 
for each $s\in [k]_0$ form a partition of $V$ into acyclic sets. 

Now, consider the graph $H$, with vertex set 
\begin{math} V(H)=\{U_1,U_2,\ldots,U_{k-1}\} \end{math} 
and $U_iU_j\in E(H)$ if and only if there are $u\in U_i$ and $v\in U_j$ such that either $(u,v)$ or $(v,u)$ is a backward arc in $D$. 
 
Let $U_iU_j\in E(H)$, then there are $u\in U_i$ and $v\in U_j$ such that 
\begin{math} l(P_u) \equiv i \pmod{k}, \end{math} 
\begin{math} l(P_v) \equiv j \pmod{k} \end{math} 
and there is a backward arc between $u$ and $v$. 
\begin{enumerate}[{Case }1:]
\item $(u,v)$ is a backward arc in $D$.
Hence, there is a $vu$-path $P$ in $T$, such that $P\cup(u,v)$ is an elementary cycle in $D$. As 
\begin{math} l(C)\equiv r_l \pmod{k} \end{math} 
for some $l\in[s]$, we have on the one hand 
\begin{math} l(P)\equiv r_l-1 \pmod{k} \end{math} 
and on the other hand 
\begin{math} l(P)=l(P_u)-l(P_v) \equiv i-j \pmod{k}; \end{math} 
and thus 
\begin{math} i-j\equiv r_l-1 \pmod{k}, \end{math} 
or equivalently 
\begin{math} i\equiv j+r_l-1 \pmod{k} \end{math} 
where 
\begin{math} i,j,r_l \in[k]_0. \end{math} 
\item $(v,u)$ is a backward arc in $D$.
Hence, there is a $uv$-path $P'$ in $T$, such that $P'\cup(v,u)$ is an elementary cycle in $D$. As 
\begin{math} l(C)\equiv r_{l'} \pmod{k} \end{math} 
for some $l'\in[s]$, we have on the one hand 
\begin{math} l(P')\equiv r_{l'}-1 \pmod{k} \end{math} 
and on the other hand 
\begin{math} l(P')=l(P_v)-l(P_u) \equiv j-i \pmod{k}; \end{math} 
and thus 
\begin{math} j-i\equiv r_{l'}-1 \pmod{k}, \end{math} 
or equivalently 
\begin{math} i\equiv j-(r_{l'}-1) \pmod{k} \end{math} 
where 
\begin{math} i,j,r_h\in[k]_0. \end{math}
\end{enumerate}
In this way, $U_i$ is adjacent to $U_j$ in $H$ if and only if 
\begin{math} i\equiv j+(r_l-1) \pmod{k} \end{math} 
or 
\begin{math} i\equiv j-(r_l-1) \pmod{k} \end{math} 
for some $l\in [s]$. 

Since there are no backward arcs with both ends in the same $U_j$ or otherwise $D$ would have a cycle of length 1 modulo $k$, we have $H$ has no loop and the neighborhood of a vertex $U_j$ in $H$ satisfies 
\begin{math} N_H(U_j)\subset \bigcup_{l=1}^s\{U_{j-(r_l-1)},U_{j+{r_l-1}}\}, \end{math} 
where the subscripts are taken modulo $k$; and thus 
\begin{math} d_H(U_j)\leq 2s \end{math} 
for each 
\begin{math} 0\leq j\leq k-1. \end{math}
Therefore, the maximum degree in $H$ satisfies $\Delta(H)\leq 2s$. By Theorem \ref{brooks' theorem}, $H$ is $(2s+1)$-colorable.

Let 
\begin{math} c'\colon V(H) \to [2s+1] \end{math} 
a proper coloring in $H$, then 
\begin{math} c\colon V \to [2s+1] \end{math} 
defined as $c(v)=c'(U_j)$ whenever $v\in U_j$ is an acyclic coloring of $D$. 

The mapping $c$ assigns a unique color to each vertex $v\in V$, as the sets $U_1$, \ldots, $U_{k-1}$ form a partition of $V$, so is well defined.

The mapping $c$ is acyclic: let $C$ be an arbitrary cycle in $D$ and $(u,v)\in A(C)$ be a backward arc, then there are different subscripts $i,j \in [k]_0$ such that $u\in U_i$, $v\in U_j$ and thus $U_iU_j \in E(H)$, in this way 
\begin{math} c(u)=c'(U_i)\neq c'(U_j)=c(v) \end{math} 
as $c'$ is proper. Hence, $C$ is not monochromatic.

Therefore, 
\begin{math} \chi_A(D)\leq 2s+1. \end{math}
\end{proof}

If we restrict $s$ in last theorem, then we improve the result in Theorem \ref{dichro theo 01}.

\begin{corollary}
\label{corollary cycles of length r_i modulo k}
Let $k$ and $s$ be two integers such that 
\begin{math} s < \lfloor{\frac{k}{2}}\rfloor; \end{math} 
let $r_1$, \ldots, $r_s$ be $s$ integers such that 
\begin{math} k\geq r_i\geq 0 \end{math} 
for each $i\in[s]$ and $D$ a strongly connected digraph such that all cycles in $D$ have length $r$ modulo $k$ for some 
\begin{math} r\in\{r_1,\ldots,r_s\}. \end{math} 
If $r_i\neq 1$ for each $i\in[s]$, then 
\begin{math} \chi_A(D)\leq 2s+1 < k. \end{math}
\end{corollary} 

\begin{corollary}
\label{corollary cycles of length r modulo k}
Let $k$ and $r$ be two integers such that 
\begin{math} k\geq r\geq 0 \end{math} 
and let $D$ be a strongly connected digraph such that every cycle in $D$ has length  $r$ modulo $k$. If $r\neq 1$, then 
\begin{math} \chi_A(D)\leq 3. \end{math}
\end{corollary}

Suppose $k$ and $r$ satisfy the hypothesis of corollary \ref{corollary cycles of length r modulo k} and take the worst case scenario in the proof of theorem \ref{theo cycles of length r_i modulo k}, where 
\begin{math} N_{H}(U_j)=\{U_{j-(r-1)}, U_{j+r-1}\} \end{math} 
for all $0\leq j\leq k-1$, where the subscripts are taken modulo $k$. Then $H$ is a union of 
\begin{math} d=\gcd(r-1,k) \end{math} 
pairwise disjoint cycles of length 
\begin{math} b=\frac{k}{d} \end{math} 
of the form 
\begin{math} C_j=U_jU_{j+r-1}U_{j+2(r-1)}\cdots U_{j+(b-1)(r-1)}U_j \end{math} 
for all $0\leq j< r-1$.**\\
If $b$ is even then, by Theorem \ref{brooks' theorem}, $H$ must be 2-colorable. 

** To clarify the statement, consider the bijective mapping 
\begin{math} \phi\colon V(H)\to \mathbb{Z}_k \end{math} 
given by 
\begin{math} U_j\mapsto \bar{j}, \end{math} 
where $\bar{n}$ is the number in $[k]_0$ such that 
\begin{math} n\equiv \bar{n} \pmod{k}. \end{math}

We know that all the backward arcs with head in $U_j$ have their tail in $U_{j+r-1}$, by Case 1 in the proof of Theorem \ref{theo cycles of length r_i modulo k} and all back arcs having tail in $U_j$ have their head in $U_{j-(r-1)}$, by Case 2 in the proof of Theorem \ref{theo cycles of length r_i modulo k}. Thus, we jump in the opposite direction of the arcs $r-1$ indices, from $U_j$ to $U_{j+r-1}$. If we consider the same jumps in $\mathbb{Z}_k$ we get the automorphism $\psi$ in $\mathbb{Z}_k$ given by 
\begin{math} \psi(\bar{n})=\overline{n+(r-1)}. \end{math}
In this way, $U_j$ and $U_l$ are adjacent in $H$ if and only if either 
\begin{math} \psi(\bar{j})=\bar{l} \end{math} 
or 
\begin{math} \psi(\bar{l})=\bar{j}. \end{math}

Notice that, there are $d$ orbits of $\psi$ which are pairwise disjoint, cyclic and have size $t$ and each orbit corresponds to precisely one cycle in $H$. The relationship between $H$ and the automorphism $\psi$ in $\mathbb{Z}_k$ is now clear. \\
Thus, determining a group of the same order as our graph (digraph) and an automorphism on it which represents the adjacencies may help us to color, via its orbits that are always cyclic and mutually disjoint, the vertices of the graph (digraph). However, if the graph (digraph) has more edges (arcs, resp.) than those in the cycles determined by the orbits, this tool may be useless.

\begin{corollary}
Let $k$ and $r$ be two integers such that $k\geq r\geq 0$  and let $D$ be a strongly connected digraph such that every cycle in $D$ has length congruent with $r$ modulo $k$. If $r\neq 1$ and $\frac{k}{d}$ is even, where 
\begin{math} d=\gcd(r-1,k), \end{math} 
then 
\begin{math} \chi_A(D)\leq 2. \end{math} 
\end{corollary}
\begin{proof}
This follows from the proof of Theorem \ref{theo cycles of length r_i modulo k} and the explanation in **, noticing that the orbits of $\phi$ as cycles in $H$ have even length 
\begin{math} b=\frac{k}{d}. \end{math}
Thus they are 2-colorable and we may color the vertices of $D$ with 2 colors as in the proof of Theorem \ref{theo cycles of length r_i modulo k}.
\end{proof}

Next we will prove Theorem \ref{theo 03}.

%\begin{theorem}
%Let $D$ be a strongly connected digraph with girth $g$; $k$ and $p$ be two integers such that 
%\begin{math} k\geq g-1 \geq p \geq 1. \end{math}
%If there is no cycle of length $r$ modulo 
%\begin{math} \hat{k}=\lceil{\frac{k}{p}}\rceil{p} \end{math} 
%in $D$ for each 
%\begin{math} r\in \{-p+2,\ldots, 0,\ldots, p\}, \end{math} 
%then 
%\begin{math} \chi_A(D)\leq \lceil{\frac{k}{p}}\rceil. \end{math}
%\end{theorem}

\begin{proof}[of Theorem \ref{theo 03}]
Let $T$ be a DFS tree of $D$ rooted at $\hat{r}$ of length $t$ and $V_0$, \ldots, $V_t$ its generations. 

Let 
\begin{math} k'= \lceil{ \frac{t+1}{p}}\rceil \end{math} 
and
\begin{math} U_{h} = \bigcup_{j=0}^{g-2} V_{hp+j} \end{math} 
for each $h\in [k']_0$, where 
\begin{math} V_{(k'-1)p+j}=\emptyset \end{math} 
whenever 
\begin{math}(k'-1)p+j>t. \end{math}
In Lemma \ref{lemma g-1 consecutive generations acyclic} we proved each $U_{h}$ is acyclic as it consists of $p\leq g-1$ consecutive generations and that there is no backward arc with both ends in the same $U_{h}$.

These sets form a partition of $V$, since the $V_i$'s form a partition of $V$ and by Euclidean algorithm each $i\in[t+1]_0$ can be written as 
\begin{math} i=hp+j \end{math} for unique $h\geq 0$ and 
\begin{math} 0\leq j \leq p-1, \end{math} 
and thus $V_i \subset U_{h'}$ and 
\begin{math} V_i\cap U_{h'}=\emptyset \end{math} 
whenever $h'\neq h$.

Define 
\begin{math} c\colon V \to [d]_0, \end{math}
where 
\begin{math} d = \lceil{\frac{k}{p}}\rceil=\frac{\hat{k}}{p}, \end{math}
as  follows: $c(v)=i$ whenever $v\in U_{h}$ and 
\begin{math} h \equiv i \pmod{d}. \end{math} 
Then $c$ is an acyclic coloring of $D$. 
First, as the sets $U_{h}$ form a partition of $V$, each vertex has exactly one color. 
Second, the color classes are acyclic. Suppose by contradiction that there is a monochromatic cycle $C$ in $D$. Then by Lemma \ref{lemma backward arc}, there is a monochromatic backward arc $(u,v)\in A(C)$. Let 
\begin{math} h_1, h_2\in [k']_0 \end{math} 
and $i\in [d]_0$ such that $u\in U_{h_1}$, $v\in U_{h_2}$ and 
\begin{math} h_1\equiv h_2\equiv i \pmod{d}. \end{math} 

As there is no backward arc with both ends in the same $U_{h}$, we have $h_1\neq h_2$.
By the definition of the $U_h$'s we have that $u\in V_{h_1 p+j_1}$ for some $j_1\in[p]_0$, $v\in V_{h_2 p+j_2}$ for some $j_2\in[p]_0$, 
\begin{math} l(P_u)=h_1p+j_1, \end{math} 
\begin{math} l(P_v)=h_2p+j_2 \end{math} 
and, as $u$ is a descendant of $v$, 
\begin{math} h_1p+j_1>h_2p+j_2, \end{math}
where 
\begin{math} h_1,h_2\in[k']_0, \end{math} 
\begin{math} j_1,j_2\in[p]_0 \end{math} 
and $h_1 > h_2$. 

Now, let $P$ be the $vu$-path in $T$. It follows that 
\begin{math} l(P)=l(P_u)-l(P_v)=(h_1p+j_1)-(h_2p+j_2)=(h_1-h_2)(g-1)+(j_1-j_2). \end{math} 
Since 
\begin{math} h_1 \equiv h_2 \pmod{d} \end{math} 
and $h_1>h_2$, it follows that 
\begin{math} h_1-h_2=qd \end{math} 
for some $q\geq 1$ and thus 
\begin{math} l(P)= (h_1-h_2)p+(j_1-j_2) = qdp + j_1-j_2 = q\hat{k} + j_1 - j_2 \equiv j_1-j_2 \pmod{\hat{k}}, \end{math} 
where 
\begin{math} |j_1-j_2|\leq p-1 \end{math} 
as 
\begin{math} j_1,j_2\in[p]_0. \end{math}

Therefore, the cycle 
\begin{math} \gamma=P\cup(u,v) \end{math} 
satisfies 
\begin{math} l(\gamma)=l(P)+1 \equiv j_1-j_2+1 \pmod{\hat{k}}, \end{math}
where 
\begin{math} - (p-1) +1 \leq j_1-j_2+1 \leq (p-1) + 1, \end{math}
or equivalently, 
\begin{math} -p+2 \leq j_1-j_2+1 \leq p, \end{math}
which contradicts the hypothesis.

Hence, $c$ is an acyclic coloring and 
\begin{math} \chi_A(D)\leq d. \end{math}
\end{proof}

As a consequence of Theorem \ref{theo 03} we have:

\begin{corollary}
Let $D$ be a strongly connected digraph with girth $g$; $k$ and $p$ be two integers such that 
\begin{math} k\geq g-1 \geq p \geq 1 \end{math} 
and $p \mid k$. If there is no cycle of length $r$ modulo $k$ in $D$ for each 
\begin{math} r\in \{-p+2,\ldots, 0,\ldots, p\}, \end{math} 
then 
\begin{math}\chi_A(D)\leq \frac{k}{p}. \end{math}
\end{corollary}

As the dichromatic number of a digraph with all of its arcs symmetric (girth $g=2$) is equivalent to the chromatic number of an undirected graph, we have as a corollary Tuza's result, Theorem \ref{tuza}.

%\begin{theorem}
%\label{circumference and girth}
%Let $D$ be a digraph with at least one cycle and let $c$ and  $g$ be its circumference and girth, respectively. Then 
%\begin{math} \chi_A(D)\leq \lceil{ \frac{c-1}{g-1}}\rceil+1. \end{math}
%\end{theorem}

Now we will prove our last main result, Theorem \ref{circumference and girth}.

\begin{proof}[of Theorem \ref{circumference and girth}]
Let $T$ be a DFS tree of $D$ rooted at $r$ with length $t$ and generations $V_0$, \ldots, $V_t$.  We will partition $V$ into 
\begin{math} k=\lceil{ \frac{c-1}{g-1}}\rceil+1 \end{math} 
acyclic subsets as follows.
\begin{enumerate}[{Step }1:]
\item Let 
\begin{math} k'= \lceil{ \frac{t+1}{g-1}}\rceil \end{math} and
\begin{math} U_{h'} = \bigcup_{j=0}^{g-2} V_{h'(g-1)+j} \end{math} 
for each $h'\in [k']_0$, where 
\begin{math} V_{(k'-1)(g-1)+j}=\emptyset \end{math} 
whenever 
\begin{math} (k'-1)(g-1)+j>t. \end{math}
In Lemma \ref{lemma partition acyclic sets} we saw that the sets $U_0$, \ldots, $U_{k'-1}$ form a partition of $V$ into acyclic sets.
\item Give color $h\in[k]_0$ to vertices in $U_{h'}$ whenever 
\begin{math} h' \equiv h \pmod{k}, \end{math}
where $h'\in [k']_0$.
The coloring is well defined as the sets $U_0$, \ldots, $U_{k'-1}$ form a partition of $V$ and $h'$ is congruent with exactly one $h\in[k]_0$ modulo $k$; thus, each vertex $v\in V$ is colored with exactly one color.
\end{enumerate}

We will see that each color class is acyclic. Suppose by contradiction that there is a monochromatic cycle $C$. By Lemma \ref{lemma backward arc}, there is a backward arc  $(u,v)\in A(C)$.  As $C$ is monochromatic, there are $h\in [k]_0$ and $i_1,i_2 \in [k']$ such that $u\in U_{i_1}$, $v\in U_{i_2}$ and 
\begin{math} i_1 \equiv i_2 \equiv h \pmod{k}. \end{math}

Observe that 
\begin{math} u\in V_{i_1(g-1)+j_1} \end{math} 
for some $j_1\in[g-1]_0$ and 
\begin{math} v\in V_{i_2(g-1)+j_2} \end{math} 
for some $j_2\in[g-1]_0$, by the definition of $U_{i_1}$ and $U_{i_2}$. Hence, 
\begin{math} l(P_u)=i_1(g-1)+j_1 \end{math} 
and 
\begin{math} l(P_v)=i_2(g-1)+j_2. \end{math}
Moreover, as $(u,v)$ is a backward arc it follows that 
\begin{math} i_1(g-1)+j_1>i_2(g-1)+j_2, \end{math}
where 
\begin{math} i_1,i_2\in[k']_0 \end{math} 
and 
\begin{math} j_1,j_2\in[g-1]_0; \end{math}
and thus $i_1 \geq i_2$. If  $i_1=i_2$, then $u,v\in U_{i_1}$ which is impossible, by Lemma \ref{lemma partition acyclic sets}. Therefore, $i_1>i_2$. Recall 
\begin{math} i_1\equiv i_2 \pmod{k}, \end{math} hence 
\begin{math} i_1-i_2=qk \end{math} 
for some $q\geq 1$.

Now, let $P$ be the $vu$-path in $T$. It follows that 
\begin{math} l(P)=l(P_u)-l(P_v)=(i_1(g-1)+j_1)-(i_2(g-1)+j_2)=(i_1-i_2)(g-1)+(j_1-j_2). \end{math} 
On the one hand, since  
\begin{math} i_1 \equiv i_2 \pmod{k}, \end{math}
we have 
\begin{math} l(P)\equiv j_1-j_2 \pmod{k} \end{math} 
where 
\begin{math} |j_1-j_2|<g-1 \end{math} 
as 
\begin{math} j_1,j_2\in[g-1]_0; \end{math}
on the other hand, as 
\begin{math} i_1-i_2=qk \end{math} 
with $q\geq 1$, we have 
\begin{math} l(P)= (i_1-i_2)(g-1)+(j_1-j_2) = qk(g-1) +(j_1-j_2)\geq k(g-1) +(j_1-j_2). \end{math}

Consider the cycle 
\begin{math} \gamma=P\cup(u,v), \end{math}
we have 
\begin{math} g\leq l(\gamma)\leq c \end{math}
and  
\begin{math} l(\gamma)=l(P)+1; \end{math}
and thus 
\begin{math} g-1\leq l(P)\leq c-1. \end{math}

However,  
\begin{math} l(P)\geq k(g-1) +(j_1-j_2)=\left(\lceil{ \frac{c-1}{g-1}}\rceil+1\right)(g-1)+(j_1-j_2) \geq \left( \frac{c-1}{g-1}+1\right)(g-1)+(j_1-j_2) = (c-1+g-1)+(j_1-j_2) > (c+g-2)-(g-1)=c-1, \end{math}
which is impossible. Hence the color classes are acyclic and we have the result.

\end{proof}

The bound of Theorem \ref{circumference and girth} is trivially reached by directed cycles (with no symmetric arcs) and complete digraphs (digraphs where each pair of different vertices induces a directed 2-cycle). And it improves Bondy's result whenever $g\geq 3$.

\acknowledgements
\label{sec:ack}

We would like to thank the anonymous referees for many comments which improved substantially the rewriting of this paper.

\nocite{*}
\bibliographystyle{abbrvnat}
% use the following instead if you encounter problems 
%\bibliographystyle{alpha}
\bibliography{Cordero-Michel_and_Galeana-Sanchez_New_bounds_for_the_dichromatic_number}

\begin{thebibliography}{13}
\providecommand{\natexlab}[1]{#1}
\providecommand{\url}[1]{\texttt{#1}}
\expandafter\ifx\csname urlstyle\endcsname\relax
  \providecommand{\doi}[1]{doi: #1}\else
  \providecommand{\doi}{doi: \begingroup \urlstyle{rm}\Url}\fi

\bibitem[Bang-Jensen and Gutin(2009)]{Bang-Jensen2009}
J.~Bang-Jensen and G.~Gutin.
\newblock \emph{Digraphs: Theory, Algorithms and Applications}.
\newblock Springer Monographs in Mathematics, Springer-Verlag London, Ltd,
  London, 2 edition, 2009.
\newblock ISBN 978-1-84800-997-4.

\bibitem[Bokal et~al.(2004)Bokal, Fijav\~{z}, Juvan, Kayll, and
  Mohar]{Bokal2004227}
D.~Bokal, G.~Fijav\~{z}, M.~Juvan, P.~Kayll, and B.~Mohar.
\newblock The circular chromatic number of a digraph.
\newblock \emph{J. Graph Theory}, 46:\penalty0 227--240, 2004.

\bibitem[Bondy(1976)]{Bondy1976277}
J.~Bondy.
\newblock Diconnected orientations and a conjecture of las vergnas.
\newblock \emph{J. London Math. Society}, 14:\penalty0 277--282, 1976.

\bibitem[Brooks(1941)]{Brooks1941194}
R.~Brooks.
\newblock On coloring the nodes of a network.
\newblock \emph{Math. Proc. of the Cambridge Philosophical Society},
  37\penalty0 (2):\penalty0 194--197, 1941.

\bibitem[Chen et~al.(2007)Chen, Hu, and Zang]{Chen2007923}
X.~Chen, X.~Hu, and W.~Zang.
\newblock A min-max theorem on tournaments.
\newblock \emph{SIAM J. Comput.}, 37:\penalty0 923--937, 2007.

\bibitem[Chen et~al.(2015)Chen, Ma, and Zang]{Chen2015210}
Z.~Chen, J.~Ma, and W.~Zang.
\newblock Coloring digraphs with forbidden cycles.
\newblock \emph{J. Comb. Theory Ser. B}, 115:\penalty0 210--223, 2015.

\bibitem[Gross and Yellen(2006)]{Gross2006}
J.~L. Gross and J.~Yellen.
\newblock \emph{Graph Theory and Its Applications}.
\newblock CRC Press, Boca Raton, FL, USA, 2 edition, 2006.
\newblock ISBN 978-1-4200-5714-0.

\bibitem[Harutyunyan and Mohar(2011)]{Harutyunyan2011170}
A.~Harutyunyan and B.~Mohar.
\newblock Gallai's theorem for list coloring of digraphs.
\newblock \emph{SIAM J. Discrete Math.}, 25:\penalty0 170--180, 2011.

\bibitem[Harutyunyan and Mohar(2012)]{Harutyunyan20121823}
A.~Harutyunyan and B.~Mohar.
\newblock Two results on the digraph chromatic number.
\newblock \emph{Discrete Math.}, 312:\penalty0 1823--1826, 2012.

\bibitem[Keevash et~al.(2013)Keevash, Li, Mohar, and Reed]{Keevash2013693}
P.~Keevash, Z.~Li, B.~Mohar, and B.~Reed.
\newblock Digraph girth via chromatic number.
\newblock \emph{SIAM J. Discrete Math.}, 27:\penalty0 693--696, 2013.

\bibitem[Mohar(2003)]{Mohar2003107}
B.~Mohar.
\newblock Circular colorings of edge-weighted graphs.
\newblock \emph{J. Graph Theory}, 43:\penalty0 107--116, 2003.

\bibitem[Neumann-Lara(1982)]{Neumann-Lara1982265}
V.~Neumann-Lara.
\newblock The dichromatic number of a digraph.
\newblock \emph{J. Comb. Theory Ser. B}, 33:\penalty0 265--270, 1982.

\bibitem[Tuza(1992)]{Tuza1992236}
Z.~Tuza.
\newblock Graph coloring in linear time.
\newblock \emph{J. Comb. Theory Ser. B}, 55:\penalty0 236--243, 1992.

\end{thebibliography}
\label{sec:biblio}

\end{document}